\documentclass[preprint,12pt,authoryear]{elsarticle}

\makeatletter
\def\ps@pprintTitle{%
  \let\@oddhead\@empty
  \let\@evenhead\@empty
  \def\@oddfoot{\hfill\thepage\hfill}%
  \let\@evenfoot\@oddfoot}
\makeatother

\usepackage{amsmath,amssymb,amsthm}
\usepackage{booktabs}
\usepackage{setspace}

\newtheorem{theorem}{Theorem}
\newtheorem{proposition}{Proposition}
\newtheorem{remark}{Remark}

\begin{document}

\doublespacing

\begin{frontmatter}

\title{Strategic Interview Positioning under Uncertain Self-Rank in the Secretary Problem}

\author[uas]{Jos\'e A. Islas\corref{cor1}}
\cortext[cor1]{Corresponding author}
\affiliation[uas]{organization={Facultad de Ciencias F\'isico-Matem\'aticas, Universidad Aut\'onoma de Sinaloa},
            city={Culiac\'an}, state={Sinaloa}, postcode={80013}, country={Mexico}}
\ead{islas.jose@uas.edu.mx}
\begin{abstract}
Under a truncated geometric self-rank prior, we study interview positioning in the secretary problem. A unique critical parameter separating first-eligible and final optimality exists exactly when $r(r+1)>n-1$; if $r_n/n\to\alpha$, the critical parameter converges to $(1-\alpha)/(1-\alpha+\alpha^2)$.
\end{abstract}

\begin{keyword}
secretary problem \sep optimal stopping \sep interview positioning \sep rank uncertainty \sep truncated geometric distribution \sep threshold transition
\end{keyword}

\end{frontmatter}

\doublespacing

\section{Introduction}

The classical secretary problem asks for a stopping rule maximizing the
probability of selecting the best applicant in random order
\citep{GilbertMosteller1966,Chow1964}. \citet{Glass2012} shifted the
viewpoint to the applicant: given the employer's rule, which interview
position maximizes hiring probability? With known absolute rank and the
employer using the classical optimal threshold, Glass characterizes the
applicant's preferred interview position.

\citet{McGrath2025} formulate a two-stage Stackelberg game in which the
employer strategically chooses a rejection threshold, anticipating the
applicant's subsequent choice of interview position. They also consider an
uninformed applicant who assigns equal probability to all possible ranks. In
our model, by contrast, the employer's threshold $r$ is treated as given rather
than optimized, while the applicant's uncertainty about absolute rank is
described by a truncated geometric prior. Conditional on a fixed absolute rank
$R=j$, the hiring probabilities in our model coincide with those in Glass's
known-rank analysis.

Fix $n\ge3$ and $1\le r\le n-2$. The employer rejects the first $r$
applicants, then accepts the first subsequent record; if no earlier selection
occurs, the final applicant is accepted. Let $R\in\{1,\ldots,n\}$ be the
applicant's absolute rank, with $R=1$ best, and model the applicant's
subjective self-rank prior by
\[
P(R=j)=p_j(q)=\frac{(1-q)q^{j-1}}{1-q^n},
\qquad j=1,\ldots,n,\quad 0<q<1.
\]
The applicant does not observe the realized rank $R$, but knows the parameter
$q$ of this truncated geometric self-rank prior. As $q\uparrow1$ the prior
tends to the uniform distribution, while $q\downarrow0$ concentrates on rank
$1$.

Geometric laws already serve as priors for uncertain discrete horizons in
secretary problems \citep{BrussSamuels1990,Tamaki2013}; here the same tractable
family models uncertainty about discrete rank. Its constant adjacent ratio
$p_{j+1}(q)/p_j(q)=q$ both measures the decay across successive ranks and
orders the family by monotone likelihood ratio. The geometric assumption is
not needed for the two-position reduction; it yields a monotone one-parameter
transition and a closed-form asymptotic boundary.

Section~\ref{sec:conditional} derives the hiring probabilities conditional on
$R=j$. The pointwise two-position comparison, which underlies Glass's
\citeyearpar{Glass2012} known-rank analysis, remains valid after averaging over
any self-rank distribution. The geometric model enters only when these
rank-conditional probabilities are averaged over the unknown rank. When it
exists, let $q_c(n,r)$ denote the finite critical value of the self-rank
parameter $q$ at which the applicant is indifferent between the first eligible
and final positions. To study large problems, suppose that for each $n$ the
employer uses a given rejection threshold $r_n$, and that
\[
\frac{r_n}{n}\to\alpha\in(0,1)
\qquad\text{as }n\to\infty.
\]
Thus, the fraction of applicants rejected before selection becomes possible
converges to the fixed proportion $\alpha$. The thresholds $r_n$ need not be optimal for the classical secretary
problem. The main result shows that, whenever the finite critical values exist,
\[
q_c(n,r_n)\to
q_c(\alpha)=\frac{1-\alpha}{1-\alpha+\alpha^2}.
\]
In particular, if $r_n^*$ denotes the classical optimal threshold, then the
standard secretary result \citep{GilbertMosteller1966} gives
\[
\frac{r_n^*}{n}\to\frac1e.
\]
Therefore, Theorem~\ref{thm:main} yields
\[
q_c(n,r_n^*)\to q_c(1/e)
=\frac{e(e-1)}{e^2-e+1}.
\]

\section{Rank-conditional analysis}
\label{sec:conditional}

There are $n\ge3$ applicants with distinct absolute ranks. Conditional on the
strategic applicant choosing position $k$, the remaining $n-1$ applicants are
uniformly randomly assigned to the remaining positions. Throughout the paper we set $\binom ab=0$ whenever $b>a$.

Let
\[
q_{j,k}
=
P(\text{applicant is hired}\mid R=j,\text{ position }k).
\]
For $k\le r$, $q_{j,k}=0$. For $r<k<n$, the applicant must be a record at
position $k$, which requires all $j-1$ superior applicants to occur after
position $k$. Hence
\[
P(\text{record at }k\mid R=j)
=
\frac{\binom{n-k}{j-1}}{\binom{n-1}{j-1}}.
\]
Conditional on this event, the employer reaches position $k$ precisely when
the best applicant among the first $k-1$ positions lies among the first $r$
positions. By symmetry this has probability $r/(k-1)$, so
\[
q_{j,k}
=
\frac{r}{k-1}
\frac{\binom{n-k}{j-1}}{\binom{n-1}{j-1}},
\qquad r<k<n.
\]
At the first eligible and final positions, respectively,
\[
q_{j,r+1}=\frac{\binom{n-j}{r}}{\binom{n-1}{r}},
\qquad
q_{j,n}=\frac{r}{n-1}.
\]

For a fixed known rank, the two-position reduction underlies Glass's
\citeyearpar{Glass2012} known-rank analysis. The following proposition shows
that the same reduction persists under arbitrary uncertainty about $R$.

\begin{proposition}[Two-position reduction]
\label{prop:two-position}
For every probability distribution on $R$, an optimal interview position can
be chosen from $\{r+1,n\}$. If $P(R=1)>0$, every position
$r+1<k<n$ is strictly worse than $r+1$.
\end{proposition}

\begin{proof}
For each fixed rank $j$ and $r+1<k<n$,
\[
q_{j,k}
=
\frac{r}{k-1}
\frac{\binom{n-k}{j-1}}{\binom{n-1}{j-1}}
\le
\frac{\binom{n-r-1}{j-1}}{\binom{n-1}{j-1}}
=
q_{j,r+1},
\]
with strict inequality for $j=1$. This pointwise comparison is the known-rank
ingredient behind the reduction. Averaging over any distribution of $R$
preserves the inequality, while positions $k\le r$ have zero hiring
probability. Hence only $r+1$ and $n$ need be compared; if $P(R=1)>0$, every
intermediate position is strictly worse than $r+1$.
\end{proof}

Thus we write
\[
V_{E,n}(q)
=
\sum_{j=1}^n p_j(q)q_{j,r+1},
\qquad
V_{L,n}
=
\frac{r}{n-1},
\]
for the hiring probabilities at Early ($r+1$) and Late ($n$),
respectively. Since $q_{j,n}=r/(n-1)$ for every rank $j$, averaging over $R$
leaves $V_{L,n}$ unchanged, so Late does not depend on the self-rank
distribution. Under the truncated geometric distribution,
\[
V_{E,n}(q)
=
\frac{1-q}
{(1-q^n)\binom{n-1}{r}}
\sum_{j=1}^{n-r}
\binom{n-j}{r}q^{j-1}.
\]

\section{Finite-$n$ threshold transition}

The truncated geometric law has full support. If $0<q_1<q_2<1$, then
\[
\frac{p_j(q_2)}{p_j(q_1)}
=
C(q_1,q_2)\left(\frac{q_2}{q_1}\right)^{j-1},
\]
where $C(q_1,q_2)>0$ is independent of $j$; hence the ratio is strictly
increasing in $j$. Thus increasing $q$ shifts probability toward larger ranks
in monotone-likelihood-ratio order. Because $q_{j,r+1}$ is nonincreasing and
nonconstant in $j$, $V_{E,n}(q)$ is strictly decreasing on $(0,1)$. Its endpoint limits are
\[
\lim_{q\downarrow0}V_{E,n}(q)=1,
\qquad
\lim_{q\uparrow1}V_{E,n}(q)=\frac1{r+1},
\]
the latter by $\sum_{j=1}^{n}\binom{n-j}{r}=\binom{n}{r+1}$.
Since $V_{L,n}=r/(n-1)$, we obtain:

\begin{theorem}[Finite-$n$ threshold transition]
\label{thm:finite}
Let $n\ge3$ and $1\le r\le n-2$.
If $r(r+1)<n-1$, Early is uniquely optimal for every $q\in(0,1)$.
If $r(r+1)=n-1$, Early is uniquely optimal for every $q\in(0,1)$ and
indifference occurs only in the limiting uniform case $q\uparrow1$.
If
\[
r(r+1)>n-1,
\]
there exists a unique $q_c(n,r)\in(0,1)$ satisfying
\[
V_{E,n}(q_c(n,r))=V_{L,n}.
\]
Early is uniquely optimal for $q<q_c(n,r)$ and Late is uniquely optimal for
$q>q_c(n,r)$. At $q=q_c(n,r)$, Early and Late are the only optimal positions.
\end{theorem}

\begin{proof}
The function $V_{E,n}(q)$ is continuous and strictly decreasing on $(0,1)$,
with endpoint limits $1$ and $1/(r+1)$, while
$V_{L,n}=r/(n-1)$ is constant. Therefore an interior crossing exists exactly when
$1/(r+1)<r/(n-1)$, equivalently $r(r+1)>n-1$, and strict monotonicity makes
the crossing unique.
If equality holds, the two values coincide only in the limiting case
$q\uparrow1$; if the inequality is reversed, Early remains strictly better.
Proposition~\ref{prop:two-position} gives the stated uniqueness of the
optimal positions.
\end{proof}

\begin{remark}[Uniform-prior endpoint]
As $q\uparrow1$, the self-rank distribution becomes uniform and
$V_{E,n}\to1/(r+1)$, whereas $V_{L,n}=r/(n-1)$. Thus, in the uniform-prior
case, Late is weakly preferred if and only if
\[
r(r+1)\ge n-1,
\qquad\text{equivalently}\qquad
r\ge\frac{\sqrt{4n-3}-1}{2}.
\]
This boundary coincides exactly with the one \citet{McGrath2025} obtain for
their uninformed applicant, who assigns equal probability to every rank;
there, the same threshold arises from comparing aggregate equilibrium payoffs
summed over all ranks in a game with an endogenously chosen employer threshold,
rather than from an exogenous $r$ and a $q\uparrow1$ limit as here. The
agreement between two structurally different derivations is a useful
cross-check on both results.
\end{remark}

When the critical value exists, cancellation of the factor $1-q_c$ gives the
exact finite-$n$ indifference equation
\begin{equation}
\label{eq:finite}
\sum_{j=1}^{n-r}
\binom{n-j}{r}q_c^{j-1}
=
\binom{n-2}{r-1}
\sum_{m=0}^{n-1}q_c^m.
\end{equation}
For $n=10$ and $r=3$, \eqref{eq:finite} gives
$q_c(10,3)\approx0.908169515$.

\section{Asymptotic critical value}

\begin{theorem}[Asymptotic critical value]
\label{thm:main}
For each $n\ge3$, let $r_n\in\{1,\ldots,n-2\}$ be a given rejection
threshold, and suppose
\[
\frac{r_n}{n}\longrightarrow\alpha,
\qquad \alpha\in(0,1),
\qquad n\to\infty.
\]
Then, for all sufficiently large $n$, there exists a unique
$q_c(n,r_n)\in(0,1)$, and
\[
\boxed{
\lim_{n\to\infty}q_c(n,r_n)
=
q_c(\alpha)
=
\frac{1-\alpha}{1-\alpha+\alpha^2}.
}
\]
For every fixed $q<q_c(\alpha)$, Early is uniquely optimal for all
sufficiently large $n$, whereas for every fixed $q>q_c(\alpha)$, Late is
uniquely optimal for all sufficiently large $n$.
\end{theorem}

\begin{proof}
For fixed $j$,
\[
q_{j,r_n+1}
=
\frac{\binom{n-j}{r_n}}{\binom{n-1}{r_n}}
=
\prod_{s=1}^{j-1}\frac{n-r_n-s}{n-s}
\longrightarrow
(1-\alpha)^{j-1}
\]
as $n\to\infty$. Fix $q\in(0,1)$ and define
\[
a_{n,j}
=
\begin{cases}
\dfrac{(1-q)q^{j-1}}{1-q^n}
\dfrac{\binom{n-j}{r_n}}{\binom{n-1}{r_n}},
&1\le j\le n,\\[2mm]
0,&j>n.
\end{cases}
\]
Then $V_{E,n}(q)=\sum_{j\ge1}a_{n,j}$ and
\[
0\le a_{n,j}\le q^{j-1}.
\]
Dominated convergence therefore gives
\[
\lim_{n\to\infty}V_{E,n}(q)
=
(1-q)\sum_{j=1}^{\infty}
[q(1-\alpha)]^{j-1}
=
\frac{1-q}{1-(1-\alpha)q}.
\]
Also,
\[
\lim_{n\to\infty}V_{L,n}
=
\lim_{n\to\infty}\frac{r_n}{n-1}
=
\alpha.
\]
Hence the limiting difference
\[
H(q)
=
\frac{1-q}{1-(1-\alpha)q}-\alpha
\]
is strictly decreasing and has the unique zero
\[
q_c(\alpha)
=
\frac{1-\alpha}{1-\alpha+\alpha^2}.
\]

Since $r_n/n\to\alpha>0$,
\[
r_n(r_n+1)>n-1
\]
for all sufficiently large $n$, so Theorem~\ref{thm:finite} yields a unique
finite zero $q_c(n,r_n)$. Let
\[
H_n(q)=V_{E,n}(q)-V_{L,n}.
\]
For every fixed $q\in(0,1)$, $H_n(q)\to H(q)$. If $\varepsilon>0$ is small
enough that
\[
0<q_c(\alpha)-\varepsilon<q_c(\alpha)+\varepsilon<1,
\]
then
\[
H(q_c(\alpha)-\varepsilon)>0,
\qquad
H(q_c(\alpha)+\varepsilon)<0.
\]
The same signs hold for $H_n$ for all sufficiently large $n$. Since $H_n$ is
strictly decreasing,
\[
q_c(\alpha)-\varepsilon
<
q_c(n,r_n)
<
q_c(\alpha)+\varepsilon.
\]
Thus
\[
\lim_{n\to\infty}q_c(n,r_n)=q_c(\alpha).
\]
The final optimality statements follow from this convergence.
\end{proof}

\section{Relation to Glass at the classical threshold}

For the classical optimal threshold sequence, $r_n^*/n\to1/e$
\citep{GilbertMosteller1966}. The rank-conditional limit used in
Theorem~\ref{thm:main} gives, for fixed $j$,
\[
q_{j,r_n^*+1}\to(1-1/e)^{j-1},
\qquad
q_{j,n}\to1/e.
\]
Hence Early is preferred when $(1-1/e)^{j-1}>1/e$. Since
\[
(1-1/e)^2>1/e,
\qquad
(1-1/e)^3<1/e,
\]
the first eligible position is preferred for $j\le3$ and the final position
for $j\ge4$, recovering Glass's Theorem~4.

\section{Discussion}

The monotonicity argument is not specific to the geometric prior. For any
positive one-parameter self-rank family ordered by monotone likelihood ratio,
$V_{E,n}$ is nonincreasing in the parameter; under strict ordering and
continuity, there is at most one critical value. The geometric family is
especially tractable because its large-$n$ average is explicit, yielding
$q_c(\alpha)$. A natural extension is to identify other self-rank families
with similarly explicit finite or asymptotic critical boundaries.


\begin{thebibliography}{6}
\expandafter\ifx\csname natexlab\endcsname\relax\def\natexlab#1{#1}\fi
\providecommand{\url}[1]{\texttt{#1}}
\providecommand{\href}[2]{#2}
\providecommand{\path}[1]{#1}
\providecommand{\DOIprefix}{doi:}
\providecommand{\ArXivprefix}{arXiv:}
\providecommand{\URLprefix}{URL: }
\providecommand{\Pubmedprefix}{pmid:}
\providecommand{\doi}[1]{\href{http://dx.doi.org/#1}{\path{#1}}}
\providecommand{\Pubmed}[1]{\href{pmid:#1}{\path{#1}}}
\providecommand{\bibinfo}[2]{#2}
\ifx\xfnm\relax \def\xfnm[#1]{\unskip,\space#1}\fi
\bibitem[{Bruss and Samuels(1990)}]{BrussSamuels1990}
\bibinfo{author}{Bruss, F.T.}, \bibinfo{author}{Samuels, S.M.},
  \bibinfo{year}{1990}.
\newblock \bibinfo{title}{Conditions for quasi-stationarity of the {Bayes} rule
  in selection problems with an unknown number of rankable options}.
\newblock \bibinfo{journal}{Annals of Probability} \bibinfo{volume}{18},
  \bibinfo{pages}{877--886}.
\newblock \DOIprefix\doi{10.1214/aop/1176990864}.
\bibitem[{Chow et~al.(1964)Chow, Moriguti, Robbins and Samuels}]{Chow1964}
\bibinfo{author}{Chow, Y.S.}, \bibinfo{author}{Moriguti, S.},
  \bibinfo{author}{Robbins, H.}, \bibinfo{author}{Samuels, S.M.},
  \bibinfo{year}{1964}.
\newblock \bibinfo{title}{Optimal selection based on relative rank}.
\newblock \bibinfo{journal}{Israel Journal of Mathematics} \bibinfo{volume}{2},
  \bibinfo{pages}{81--90}.
\newblock \DOIprefix\doi{10.1007/BF02759948}.
\bibitem[{Gilbert and Mosteller(1966)}]{GilbertMosteller1966}
\bibinfo{author}{Gilbert, J.P.}, \bibinfo{author}{Mosteller, F.},
  \bibinfo{year}{1966}.
\newblock \bibinfo{title}{Recognizing the maximum of a sequence}.
\newblock \bibinfo{journal}{Journal of the American Statistical Association}
  \bibinfo{volume}{61}, \bibinfo{pages}{35--73}.
\newblock \DOIprefix\doi{10.1080/01621459.1966.10502008}.
\bibitem[{Glass(2012)}]{Glass2012}
\bibinfo{author}{Glass, D.B.}, \bibinfo{year}{2012}.
\newblock \bibinfo{title}{The secretary problem from the applicant's point of
  view}.
\newblock \bibinfo{journal}{College Mathematics Journal} \bibinfo{volume}{43},
  \bibinfo{pages}{76--81}.
\newblock \DOIprefix\doi{10.4169/college.math.j.43.1.076}.
\bibitem[{McGrath and Schr{\"o}der(2025)}]{McGrath2025}
\bibinfo{author}{McGrath, T.}, \bibinfo{author}{Schr{\"o}der, M.},
  \bibinfo{year}{2025}.
\newblock \bibinfo{title}{Competitive secretary problem}.
\newblock \bibinfo{journal}{International Journal of Game Theory}
  \bibinfo{volume}{54}, \bibinfo{pages}{1}.
\newblock \DOIprefix\doi{10.1007/s00182-025-00931-9}.
\bibitem[{Tamaki(2013)}]{Tamaki2013}
\bibinfo{author}{Tamaki, M.}, \bibinfo{year}{2013}.
\newblock \bibinfo{title}{Optimal stopping rule for the no-information duration
  problem with random horizon}.
\newblock \bibinfo{journal}{Advances in Applied Probability}
  \bibinfo{volume}{45}, \bibinfo{pages}{1028--1048}.
\newblock \DOIprefix\doi{10.1239/aap/1386857856}.

\end{thebibliography}
\end{document}